\documentclass[10pt]{amsart}
\usepackage{pdfsync}
\usepackage{latexsym}
\usepackage{amsgen,amsmath,amsxtra,amssymb,amsfonts,amsthm}

\usepackage{graphicx}
\usepackage[dvips]{epsfig}
\usepackage{subfigure}
\usepackage[active]{srcltx}

\begin{document}
\newcommand{\R}{{\mathbb R}}
 \newcommand{\Hi}{{\mathbb H}}
\newcommand{\Ss}{{\mathbb S}}
\newcommand{\N}{{\mathbb N}}
\newcommand{\Rn}{{\mathbb{R}^n}}
\newcommand{\ieq}{\begin{equation}}
\newcommand{\eeq}{\end{equation}}
\newcommand{\ieqa}{\begin{eqnarray}}
\newcommand{\eeqa}{\end{eqnarray}}
\newcommand{\ieqas}{\begin{eqnarray*}}
\newcommand{\eeqas}{\end{eqnarray*}}
\newcommand{\Bo}{\put(260,0){\rule{2mm}{2mm}}\\}
\def\L#1{\label{#1}} \def\R#1{{\rm (\ref{#1})}}


\theoremstyle{plain}
\newtheorem{theorem}{Theorem} [section]
\newtheorem{corollary}[theorem]{Corollary}
\newtheorem{lemma}[theorem]{Lemma}
\newtheorem{proposition}[theorem]{Proposition}


\theoremstyle{definition}
\newtheorem{definition}[theorem]{Definition}
\newtheorem{remark}[theorem]{Remark}

\numberwithin{figure}{section}
\newcommand{\res}{\mathop{\hbox{\vrule height 7pt width .5pt depth
0pt \vrule height .5pt width 6pt depth 0pt}}\nolimits}
\def\at#1{{\bf #1}: } \def\att#1#2{{\bf #1}, {\bf #2}: }
\def\attt#1#2#3{{\bf #1}, {\bf #2}, {\bf #3}: } \def\atttt#1#2#3#4{{\bf #1}, {\bf #2}, {\bf #3},{\bf #4}: }
\def\aug#1#2{\frac{\displaystyle #1}{\displaystyle #2}} \def\figura#1#2{ \begin{figure}[ht] \vspace{#1} \caption{#2}
\end{figure}} \def\B#1{\bibitem{#1}} \def\q{\int_{\Omega^\sharp}}
\def\z{\int_{B_{\bar{\rho}}}\underline{\nu}\nabla (w+K_{c})\cdot
\nabla h} \def\a{\int_{B_{\bar{\rho}}}}
\def\b{\cdot\aug{x}{\|x\|}}
\def\n{\underline{\nu}} \def\d{\int_{B_{r}}}
\def\e{\int_{B_{\rho_{j}}}} \def\LL{{\mathcal L}}
\def\itr{\mathrm{Int}\,}
\def\D{{\mathcal D}}
 \def\tg{\tilde{g}}
\def\A{{\mathcal A}}
\def\S{{\mathcal S}}
\def\H{{\mathcal H}}
\def\M{{\mathcal M}}
\def\T{{\mathcal T}}
\def\U{{\mathcal U}}
\def\N{{\mathcal N}}
\def\I{{\mathcal I}}
\def\F{{\mathcal F}}
\def\J{{\mathcal J}}
\def\E{{\mathcal E}}
\def\F{{\mathcal F}}
\def\G{{\mathcal G}}
\def\HH{{\mathcal H}}
\def\W{{\mathcal W}}
\def\H{\D^{2*}_{X}}
\def\d{d^X_M }
\def\LL{{\mathcal L}}
\def\H{{\mathcal H}}
\def\HH{{\mathcal H}}
\def\itr{\mathrm{Int}\,}
\def\vah{\mbox{var}_\Hi}
\def\vahh{\mbox{var}_\Hi^1}
\def\vax{\mbox{var}_X^1}
\def\va{\mbox{var}}
\def\SS{{\mathcal S}}
 \def\Y{{\mathcal Y}}
\def\length{{l_\Hi}}
\newcommand{\average}{{\mathchoice {\kern1ex\vcenter{\hrule
height.4pt width 6pt depth0pt} \kern-11pt} {\kern1ex\vcenter{\hrule height.4pt width 4.3pt depth0pt} \kern-7pt} {} {} }}

\newcommand{\ave}{\average\int}

\title[Some fourth order nonlinear elliptic problems]{Some fourth order nonlinear elliptic problems related to epitaxial growth}

\author[C. Escudero, I. Peral]{Carlos Escudero, Ireneo Peral}\thanks{Work partially supported by project
MTM2010-18128, MINECO, Spain. C. E. also supported by RYC-2011-09025, MINECO, Spain}

\address{Departamento de Matem\'{a}ticas,  Universidad Aut\'{o}noma de Madrid, 28049 Madrid, Spain.}

\email{}

\keywords{Growth problems, higher order elliptic equations, Gaussian curvature, Monge-Amp\`{e}re type equations, existence of solutions,
variational methods.\\ \indent 2010 {\it MSC:  35J50, 35J60, 35J62, 35J96, 35G20, 35G30.}}

\date{\today}

\begin{abstract}
This paper deals with some mathematical models arising in the theory of epitaxial growth of crystal. We focalize the study on a stationary
problem which presents some analytical difficulties. We study the existence  of solutions. The central model in this work is given by the
following fourth order elliptic equation,
$$\begin{array}{rclll}
\Delta^2 u=\text{det} \left( D^2 u \right) &+&\lambda f, \quad & x\in \Omega\subset\mathbb{R}^2\\ \hbox{ conditions  on} &\quad&
 & \partial \Omega.
\end{array}
$$
The framework to study the problem deeply depends on the boundary conditions.
\end{abstract}

\maketitle


\rightline{\it To the memory of James Serrin.}

\section{Introduction}

In this work we are concerned with the stationary version of (\ref{parabolic2}) below, which reads
\begin{equation}\label{Pro0}
\left\{\begin{array}{rcl} \Delta^2 u&=&\text{det} \left( D^2 u \right) +\lambda f, \qquad x\in \Omega\subset\mathbb{R}^2, \\
\text{boundary}&\,& \text{ conditions,}
\end{array}
\right.
\end{equation}
where $\Omega$ has smooth boundary, $n$ is the unit outward normal to $\partial \Omega$,  $f$ is a function with a suitable hypothesis of
summability and $\lambda>0$. We will concentrate on Dirichlet boundary conditions, that is,
\begin{equation}\label{dir}
u=0, \quad \dfrac{\partial u}{\partial n} =0\hbox{ on }\partial \Omega
\end{equation}
and Navier conditions
\begin{equation}\label{nav}
u=0, \quad  \Delta u=0\hbox{ on }\partial \Omega,
\end{equation}

This type of  problems appears in a model of {\it epitaxial growth}.

Epitaxial growth is characterized by the deposition of new material on existing layers of the same material under high vacuum conditions.
This technique is used in the semiconductor industry for the growth of thin films~\cite{barabasi}. The mathematical description of epitaxial
growth uses the function
\begin{equation}
u: \Omega \subset \mathbb{R}^2 \times \mathbb{R}^+ \rightarrow \mathbb{R},
\end{equation}
which describes the height of the growing interface at the spatial point $x \in \Omega \subset \mathbb{R}^2$ at time $t \in \mathbb{R}^+$. A
basic modelling assumption is of course that $u$ is an univalued function, a fact that holds in a reasonably large number of
cases~\cite{barabasi}. The macroscopic description of the growing interface is given by a partial differential equation for $u$ which is
usually postulated using phenomenological and symmetry arguments~\cite{barabasi,marsili}.

We will focus on one such an equation that was derived in the context of non-equilibrium surface growth~\cite{escudero}. It reads
\begin{equation} \label{parabolic2}
u_t = 2 \, K_1 \, \det \left( D^2 u \right) - K_2 \, \Delta^2 u + \xi(x,t).
\end{equation}
The field $u$ evolves in time as dictated by three different terms: linear, nonlinear and non-autonomous ones. Both linear and nonlinear
terms describe the dynamics on the interface. The nonlinear term is the small gradient expansion (which assumes $|\nabla u| \ll 1$) of the
Gaussian curvature of the graph of $u$. The linear term is the bilaplacian, that can be considered as the linearized Euler-Lagrange equation
of the Willmore functional~\cite{hornung}. This is in fact, see below, the proper interpretation in this context. Finally, the non-autonomous
term takes into account the material being deposited on the growing surface. We can consider this equation as a sort of Gaussian curvature
flow~\cite{chow,andrews} which is stabilized by means of a higher order viscosity term. The derivation of this equation was actually
geometric: it arose as a gradient flow pursuing the minimization of the functional
\begin{equation}\label{funcional de partida}
\mathcal{V}(u)= \int_\Omega \left( K_1 H + \frac{K_2}{2} H^2 \right)\sqrt{1+|\nabla u|^2}\, dx,
\end{equation}
where $H$ is the mean curvature of the graph of $u$, in the context of non-equilibrium statistical mechanics of surface
growth~\cite{marsili,escudero}. The actual terms on the right hand side of~(\ref{parabolic2}) are found after formally expanding the
Euler-Lagrange equation corresponding to this functional for small values of the different derivatives of $u$ and retaining only linear and
quadratic terms. This type of derivation does not rely on a detailed modelling of the processes taking place in a particular physical system.
Instead, because its goal is describing just the large scale properties of the growing interface, it is based on rather general
considerations. This is the usual way of reasoning in this context, in which equations are derived as representatives of universality
classes, this is, sets of physical phenomena sharing the same large scales properties~\cite{barabasi}. In this sense,
equation~(\ref{parabolic2}) has been claimed to be an exact representative of a well known universality class within the realm of
non-equilibrium growth~\cite{escudero2}. In particular, frequently used models of epitaxial growth belong to this universality
class~\cite{escudero}.

 We will devote this work to obtain some results on existence and multiplicity of solutions to  the fourth order elliptic problem
 \eqref{Pro0}.
  One of the main consequences of this work is to exhibit the dependence on the boundary conditions, even in  the formulation of the
  problems.

More precisely the organization of the paper is as follows.

In Section \ref{preliminar} we formulate some known results on properties of the Hessian of a function in the Sobolev space
$W^{2,2}(\mathbb{R}^N)$, that will be used in the article.

Section \ref{existenceresults} is devoted to the variational formulation of the problem with Dirichlet boundary conditions and proving, via
critical points arguments, existence and multiplicity of solutions. We will use some arguments on minimization and  a version by Ekeland, in
\cite{E} (see also \cite{AE}) of the Ambrosetti-Rabinowitz {\it Mountain Pass Theorem} (see \cite{AR}).

The Navier   boundary conditions, as far as we know, have not a variational formulation. In Section \ref{nbc} we prove existence of solution
for the problem with Navier conditions by using some fixed point arguments. Finally, in Section \ref{further}, we present some related
remarks and open problems.

\section{Preliminaries.Functional setting}\label{preliminar}

If  $v$ is a smooth function we have the following chain of equalities,
\begin{eqnarray*}
 &\text{det} \left( D^2 v \right) = v_{x_1 x_1}v_{x_2 x_2}- v_{x_1 x_2}^2=
(v_{x_1}v_{x_2 x_2})_{x_1}-(v_{x_1}v_{x_2 x_1})_{x_2}=\\
 &(v_{x_1}v_{x_2})_{x_1 x_2}-\frac 12 (v_{x_2}^2)_{x_1x_1}-\frac 12 (v_{x_1}^2)_{x_2 x_2}.
\end{eqnarray*}
From now on we will assume $\Omega \subset \mathbb{R}^2$ is open, bounded and has a smooth boundary. Notice that by density we can consider
the above identities in $\mathcal{D}'(\Omega)$, the space of distributions. This  subject is deeply  related with a conjecture by J. Ball \cite{ball}:

{\it  If    $ u=(u^1,u^2)\in W^{1,p}(\Omega, \mathbb{R}^2)$  consider $\textrm{det} \left( D u \right)=u^1_{x_1}u^2_{x_2}-u^1_{x_2}u^2_{x_1}$
and define $$\textrm{Det}(D u)=(u^1u^2_{x_2})_{x_1}-(u^1u^2_{x_1})_{x_2}.$$

 When is it true that  $\textrm{det} \left( D u \right) = \textrm{Det} \left( D u \right) ?$}

A positive answer, among others results,  was given by   S. M\"{u}ller in \cite{muller}.

We will use Theorem VII.2, page 278  in \cite{coifman}, that we formulate as follows.
\begin{lemma}\label{clms}
Let $v \in W^{2,2}(\mathbb{R}^2)$. Then,
$$\text{det} \left(D^2 v \right),$$
$$(v_{x_1}v_{x_2 x_2})_{x_1}-(v_{x_1}v_{x_2 x_1})_{x_2}$$
and
$$(v_{x_1}v_{x_2})_{x_1 x_2}-\frac 12 (v_{x_2}^2)_{x_1 x_1}-\frac 12 (v_{x_1}^2)_{x_2 x_2}$$
belong to the space $\mathcal{H}^1(\mathbb{R}^2)$ and are equal in it, where $\mathcal{H}^1(\mathbb{R}^2)$ is the Hardy space.
\end{lemma}
All the expressions  involving third derivatives are understood in the distributional sense; then the result in the lemma is highly
non-trivial  and for the proof we refer to \cite{coifman}. It is interesting to point out that this result deeply depends on Luc Tartar and Fran\c{c}ois Murat arguments on {\it
compactness by compensation} \cite{tartar1,tartar2} and \cite{murat} respectively.

For the reader convenience we recall the definition of the Hardy space in $\mathbb{R}^N$ (see E. M. Stein G. Weiss \cite{SteWe}).

\begin{definition}\label{hardyspace}  The Hardy space in $\mathbb{R}^N$ is defined in an equivalent way as follows
$$\begin{array}{rcl}
\mathcal{H}^1(\mathbb{R}^N)&=&\{ f\in L^1(\mathbb{R}^N)\,|\, R_j(f)\in L^1(\mathbb{R}^N),\, j = 1, 2, \cdots, N\}=\\ &&\{ f\in
L^1(\mathbb{R}^N)\,|\,\sup|f*h_t(x)|\in L^1(\mathbb{R}^N)\},
\end{array}$$
where $R_j$ is the classical Riesz transform, that is,
$$R_j=\frac{\partial}{\partial x_j}(-\Delta )^{\frac 12}, \,\, j = 1, 2, \cdots, N,$$
and
$$h_t(x)=\frac 1{t^N} \, h \left(\frac{x}{t} \right),\hbox{  where   } h\in \mathcal{C}_0^\infty( \mathbb{R}^N),\,\, h(x)\ge 0 \hbox{ and
}\int_{\mathbb{R}^N} h \, dx =1.$$
\end{definition}
Notice  that, as a direct consequence of the the definition, if $f\in \mathcal{H}^1(\mathbb{R}^N)$ and $E_{1,N}(x)$ is the fundamental
solution to the Laplacian  in  $\mathbb{R}^N$, then
$$u(x)=\int_{\mathbb{R}^N} E_{1,N}(x-y)f(y)dy $$
verifies that $u\in W^{2,1}(\mathbb{R}^N)$. See for instance \cite{Stein}.

In a similar way if  we consider $E_{2,N}(x)$, the fundamental solution to $\Delta^2$  in  $\mathbb{R}^N$, a direct calculation shows that
$D^{\alpha}E(x)$, $|\alpha|=4$,  are of the form
$$D^{\alpha}E_{2,N}(x)=\dfrac {H_\alpha(\bar{x})}{|x|^N}, \quad \bar{x}=\dfrac{x}{|x|}$$
where $H_\alpha$ is a positively homogeneous function of zero degree and
$$\int_{S^{N-1}} H_\alpha(\bar{x}) d \bar{x}=0,$$
that is,  $D^{\alpha}E_{2,N}(x)$ is a classical Calderon-Zygmund kernel. For $f\in \mathcal{H}^1(\mathbb{R}^N)$  consider
$$u(x)=\int_{\mathbb{R}^N} E_{2,N}(x-y)f(y)dy, $$
then $u\in W^{4,1}(\mathbb{R}^N)$.

D. C. Chang, G. Dafni, and E. M. Stein in \cite{D} give the definition  of Hardy space in a bounded domain $\Omega$ in order to have the
regularity theory for the Laplacian similar to the one in $\mathbb{R}^N$.

The extension of this kind of regularity result  to the bi-harmonic  equation on bounded domains is of interest for the current problem. The
application of such a result would allow obtaining extra regularity in the nonlinear setting studied in this work. We leave these questions
as a subject of future research.

We use the following consequence which is a by product of the results in \cite{coifman} and in \cite{D}.
\begin{lemma}\label{smomega}
Let $u \in W^{2,2}_0(\Omega)$. Then
$$
\text{Det} \left(D^2 u \right)=(u_{x_1}u_{x_2})_{x_1 x_2}-\frac 12 (u_{x_2}^2)_{x_1 x_1}-\frac 12 (u_{x_1}^2)_{x_2 x_2}
$$
in $L^1(\Omega)\cap h^1_r(\Omega)$. Here $h^1_r(\Omega)$ is the class of function restrictions of  $\mathcal{H}^1(\mathbb{R}^N)$ to
$\Omega$.
\end{lemma}

\section{Variational settings: existence  and multiplicity results for the Dirichlet conditions}
\label{existenceresults}

We will study the following problem,
\begin{equation}\label{ProD}
\left\{\begin{array}{rcl} \Delta^2 u&=&\text{det} \left( D^2 u \right) +\lambda f, \qquad x\in \Omega\subset\mathbb{R}^2, \\ \text{u=0}, &\,&
\dfrac{\partial u}{\partial n}=0\hbox{ on }\partial \Omega,
\end{array}
\right.
\end{equation}
this is, Dirichlet boundary conditions, where $\Omega$ is a bounded domain with smooth boundary and $f\in L^1(\Omega)$. The natural framework
is the space $W^{2,2}_0(\Omega)$, that is the completion of $\mathcal{C}_0^\infty (\Omega)$ with the norm of $W^{2,2}(\Omega)$. This fact  is
the key to have in this case a variational formulation of the problem. We recall that the norm of the Hilbert space $W^{2,2}_0(\Omega)$ is
equivalent to the norm $\|\Delta u\|_2$. For all the functional framework we refer the reader to the precise, detailed and very nice
monograph by F. Gazzola, H. Grunau and G. Sweers, \cite{GGS}.  See \cite{ADN1} and \cite{ADN2} as classical references for elliptic equations
of higher order.
\begin{remark} We could consider the inhomogeneous Dirichlet problem,
\begin{equation}\label{ProDI}
\left\{\begin{array}{rcl} \Delta^2 u&=&\text{det} \left( D^2 u \right) +\lambda f, \qquad x\in \Omega\subset\mathbb{R}^2, \\ u=\mu_1, &\,&
\dfrac{\partial u}{\partial n}=\mu_2\hbox{ on }\partial \Omega,
\end{array}
\right.
\end{equation}
where $\mu_i$, $i=1,2$ are in  suitable trace spaces. Considering the bi-harmonic function $w$ with data $\mu_1$ and $\mu_2$ and defining
$v=u-w$ we reduce the problem to the homogeneous  case with a linear perturbation and a new source term, that is,
$$
\left\{\begin{array}{lll} \Delta^2 v=\text{det} \left( D^2 v \right) +w_{yy}v_{xx}+w_{xx}v_{yy}-2w_{xy}v_{xy} + \text{det}\left( D^2 w \right) + \lambda f, \, x\in \Omega\subset\mathbb{R}^2, \\
\text{v=0}, \, \dfrac{\partial v}{\partial n}=0\hbox{ on }\partial \Omega.
\end{array}
\right.
$$
This problem can be solved by using  fixed point arguments for small data. See the next Section \ref{nbc}. The smallness of the data is only known so far in the radial framework. See \cite{Los4}.
\end{remark}

We try to find a Lagrangian $L(\nabla u, D^2 u)$ such that the critical points of the functional
\begin{equation}\label{D}
\left\{\begin{array}{rcl} J_\lambda: &W^{2,2}_0(\Omega)&\rightarrow \mathbb{R} \\& u& \rightarrow J_\lambda(u)=\displaystyle
\frac12\int_\Omega |\Delta u|^2 \, dx -\int_\Omega L(\nabla u, D^2 u) \, dx
\end{array}
\right.
\end{equation}
are solutions to \eqref{ProD}.

\subsection{Lagrangian for the Dirichlet conditions}

We will try to obtain a Lagrangian for which the Euler  first variation is the determinant of the Hessian matrix. Our ingredients will be the
distributional identity
$$\text{det} \left(D^2 v \right)=(v_{x_1}v_{x_2})_{x_1 x_2}-\frac 12 (v_{x_2}^2)_{x_1 x_1}-\frac 12 (v_{x_1}^2)_{x_2 x_2}$$
and the fact that $\mathcal{C}^\infty_0(\Omega)$ is dense in $W^{2,2}_0(\Omega)$.

Consider $\phi\in\mathcal{C}_0^\infty(\Omega)$ and
\begin{eqnarray} \nonumber
\int_\Omega \det \left( D^2 u \right) \, \phi \,\, dx &=&\int_\Omega \left[ -\frac12 \, (u_{x_2}^2)_{x_1 x_1} - \frac12 \,
(u_{x_1}^2)_{x_2 x_2} + (u_{x_1} u_{x_2})_{x_1 x_2} \right] \phi \,\, dx \nonumber \\ \nonumber \\ \nonumber &=& \int_\Omega \left[ \frac12
\, \phi_{x_1} (u_{x_2}^2)_{x_1} + \frac12 \, \phi_{x_2} (u_{x_1}^2)_{x_2} + u_{x_1} u_{x_2} \phi_{x_1 x_2} \right] dx
\\ \nonumber \\
&=&\nonumber \left. \frac{d  }{dt}G(u + t \phi) \right|_{t=0},
\end{eqnarray}
where
\begin{equation}
G(u) := \displaystyle \int_\Omega u_{x_1} u_{x_2} u_{x_1 x_2} \, dx.
\end{equation}
Notice that by density we can take $\phi \in W^{2,2}_0(\Omega)$ and by direct application of Lemma 2.1 above we find that the first variation
of  $G(u)$ on $W^{2,2}_0(\Omega)$ is
$$\dfrac{\delta G(u)}{\delta u}= \det\left( D^2 u \right).$$

Then we will consider as {\it energy functional} for problem \eqref{ProD} the following one
\begin{equation}\label{energy}
J_\lambda(u)=\displaystyle \frac12\int_\Omega |\Delta u|^2 \, dx -\int_\Omega u_{x_1} u_{x_2} u_{x_1 x_2} \, dx - \lambda \int_\Omega f u \,
dx,
\end{equation}
defined in   $W^{2,2}_0(\Omega)$.

As we will see $J_\lambda$ is unbounded from below and then we cannot use standard minimization results but the general theory of critical
points of functionals.

\begin{remark} Notice that this Lagrangian is not useful for other boundary conditions. Indeed, consider $\phi\in \mathcal{X}=\{\phi\in
\mathcal{C}^\infty(\Omega)\,|\, \phi(x)=0 \hbox{   on   } \partial\Omega\}$ and $u$ a smooth function, then
$$\begin{array}{rcl}
\dfrac{d}{dt} G(u+t\phi)|_{t=0}&=&\displaystyle\int_\Omega \big(  u_{x_1} u_{x_2} \phi_{x_1 x_2}+u_{x_1} \phi_{x_2} u_{x_1 x_2}+\phi_{x_1}
u_{x_2} u_{x_1 x_2} \big)\,dx\\ &&\\
                                      &=&\displaystyle\int_\Omega \det\left( D^2 u \right)\phi\, dx-\dfrac 12\int_{\partial\Omega}
                                      u_{x_1}u_{x_2}\big(\phi_{x_1} \nu_{x_2} +
                                      \phi_{x_2} \nu_{x_1}\big)ds,
\end{array}
$$
and therefore for $\phi\in \mathcal{X}$ the boundary term does not cancel.

This observation justifies  the dependence of the problem on the boundary conditions.
\end{remark}

\subsection{The geometry of $J_\lambda$}
\label{geom}

Notice that by H\"older and Sobolev inequalities we find the following estimate
$$\begin{array}{lllllllll} & \displaystyle J_\lambda(u)\ge
\\ & \displaystyle \quad \\ & \displaystyle \frac 12 \int_\Omega
|\Delta u|^2 \, dx -\left(\int_\Omega |u_{x_1 x_2}|^2 \,dx \right)^{\frac 12} \left(\int_\Omega |u_{x_1}|^4 \, dx \right)^{\frac 14}
\left(\int_\Omega |u_{x_2}|^4 \, dx \right)^{\frac 14} \\ & \displaystyle \quad \\ & \displaystyle -\lambda ||f||_1||u||_\infty \ge \\ &
\displaystyle \quad \\ & \displaystyle \frac 12 \int_\Omega |\Delta u|^2 \, dx -c_1 \left( \int_\Omega |\Delta u|^2 \, dx \right)^{\frac 32}
-\lambda c_2||f||_1 \left(\int_\Omega |\Delta u|^2 \, dx \right)^{\frac 12}
\\ & \displaystyle \quad \\ & \displaystyle \equiv g \left(||\Delta u||_2 \right),
\end{array}$$
where
\begin{equation}\label{g}
g(s)= \frac 12 \, s^2 - c_1 \, s^3 - \lambda \, c_2 \,||f||_1 \,  s.
\end{equation}
Therefore we easily prove that for $0<\lambda<\lambda_0$ small enough, the {\it radial lower estimate (in the Sobolev space),} given by $g$
has a negative local minimum and a positive local maximum. Moreover, it is easy to check that:
\begin{enumerate}
\item There exists a function $\phi\in W^{2,2}_0(\Omega)$ such that
$$\int_\Omega f\phi \, dx >0.$$
\item There exists a function $\psi\in W^{2,2}_0(\Omega)$ such that
$$\int_\Omega \psi_{x_1} \psi_{x_2} \psi_{x_1 x_2}\, dx >0.$$
\end{enumerate}
For the function $\phi$ we just need it to be a local mollification of $f$. The case of $\psi$ is a bit more involved but one still has many
possibilities such as $$\psi=[(1-|x|^2)^+]^4,$$ where $|x|=\sqrt{x_1^2 + x_2^2}$ and $(\cdot)^+=\max\{\cdot,0\}$, that fulfils the positivity
criterion even pointwise in a domain containing the unit ball. Then, in general, if $B_{2r}(x_0)\subset\Omega$ we consider
$\psi_\Omega(x)=\psi(\dfrac{x-x_0}{r})$.

Other suitable functions can be found by means of deforming this one adequately. Notice that the $\psi$ function we have chosen is in
$C^2(\mathbb{R}^2)$.

According to the previous remark we find that
$$J_\lambda(t\phi)<0 \hbox{ for $t$ small enough and } J_\lambda(s\psi)<0 \hbox{ for $s$ large enough}.$$
\begin{figure}\label{fig:2.1}
\begin{center}
\includegraphics[scale=0.3]{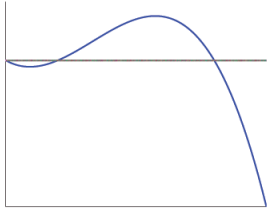}
\end{center}
\caption{The {\it radial profile} and the properties proven in the text show that the mountain pass geometry holds for $J_\lambda$ if
$\lambda>0$ is small enough. Moreover a {\it local} minimum could be found.}
\end{figure}
This behavior and the radial minorant (see Figure 3.1), suggest a kind of {\it mountain pass} geometry. See the classical paper by A.
Ambrosetti and P.~H. Rabinowitz, \cite{AR}.

\subsection{Palais-Smale condition for  $J_\lambda$}
\label{topol}

As usual,  we call $\{u_k\}_{k\in \mathbb{N}}\subset W^{2,2}_0(\Omega)$ a Palais-Smale sequence for $J_\lambda$ to the level $c$ if
\begin{itemize}
\item[i)] $J_\lambda(u_k)\to c$ as $k\to\infty$ \item[ii)] $J_\lambda'(u_k)\to 0$ in $W^{-2,2}(\Omega)$.
\end{itemize}
We say that $J_\lambda$ satisfies the local Palais-Smale condition to the level $c$ if each Palais-Smale sequence to the level $c$,
$\{u_k\}_{k\in \mathbb{N}}$, admits a strongly convergent subsequence in $W^{2,2}_0(\Omega).$

We are able to prove the following compactness  result.
\begin{lemma}\label{lemma:PS}
Assume a bounded Palais-Smale condition for $J_\lambda$, that is $\{u_k\}_{k\in \mathbb{N}}\subset W^{2,2}_0(\Omega)$ verifying
\begin{enumerate}
\item $J_\lambda(u_k)\to c$ as $k\to\infty$, \item $J_\lambda'(u_k)\to 0$ in $W^{-2,2}$.
\end{enumerate}
Then there exists a subsequence $\{u_k\}_{k\in \mathbb{N}}$ that converges in $W^{2,2}_0(\Omega)$.
\end{lemma}
\begin{proof} Since $\{u_k\}_{k\in \mathbb{N}}\subset W^{2,2}_0(\Omega)$ is bounded,
 up to passing to a subsequence, we have:
\begin{itemize}
\item[$i)$] $u_k\rightharpoonup u$ weakly in $W^{2,2}_0(\Omega)$, \item[$ii)$]  $\nabla u_k\to \nabla u$ strongly in
    $[L^p(\Omega)]^2$ for all $p<\infty$, \item[$iii)$]$u_k\to u$ uniformly in $\Omega$.
\end{itemize}
We could write the condition $J_\lambda'(u_k)\to 0$ in $W^{-2,2}$ as
\begin{equation}\label{aproxsol}
\Delta^2 u_k=\det(D^2 u_k)+ \lambda f+y_k,\,\, u_k\in W^{2,2}_0(\Omega)\hbox{ and   } y_k\to 0 \hbox{  in  } W^{-2,2}(\Omega).
\end{equation}
Notice that multiplying \eqref{aproxsol} by $(u_k-u)$, we have for all fixed $k$
\begin{equation}\label{aproxsol1}
\begin{array}{lll}
\displaystyle \int_\Omega \Delta (u_k) \Delta(u_k-u)  \, dx = \\ \quad \\  \displaystyle \int_\Omega (u_k-u) \, \text{det} \left( D^2 u_k
\right) \, dx+ \lambda \int_\Omega f \, (u_k-u) \, dx + y_k(u_k-u).
\end{array}
\end{equation}
The three terms on the right hand side go to zero as $k\to \infty$ by the convergence properties $i)$ and $iii)$. Moreover adding in both
terms of \eqref{aproxsol1}
$$-\int_\Omega \Delta u \, \Delta(u_k-u)  \, dx=o(1) \quad k\to \infty,$$
we obtain,
$$
\begin{array}{lll}
\displaystyle \int_\Omega\left| \Delta (u_k-u) \right|^2  \, dx =
\\ \quad \\  \displaystyle \int_\Omega (u_k-u) \, \text{det}
\left( D^2 u_k \right) \, dx+ \lambda \int_\Omega f \, (u_k-u) \, dx \, +  \\ \quad
\\ \displaystyle
y_k(u_k-u)-\int_\Omega \Delta u \, \Delta(u_k-u) \, dx.
\end{array}
$$
As a consequence
\begin{equation}\label{PSC}
\int_\Omega\left| \Delta (u_k-u) \right|^2  \, dx \to 0 \hbox{
    as     } k\to \infty,
\end{equation}
that is, $J_\lambda$ satisfies the Palais-Smale condition to the level $c$.
\end{proof}

\subsection{The main multiplicity result}

We now can  prove the existence and multiplicity result.

\begin{theorem}\label{existence}
Let $\Omega\subset\mathbb{R}^2$ be a bounded domain with smooth boundary. Consider $f\in L^1(\Omega)$ and $\lambda>0$. Then there exists a
$\lambda_0$ such that for $0<\lambda<\lambda_0$ problem \eqref{ProD} has at least two solutions.
\end{theorem}

\begin{proof} By the Sobolev embedding theorem the functional
$J_\lambda$ is well defined in $W^{2,2}_0(\Omega)$, is continuous and Gateaux differentiable, and its derivative is weak-* continuous
(precisely the regularity required in the weak version by Ekeland  of the  {\it  mountain pass theorem} in \cite{AE}).

We will try to prove the existence of a solution which corresponds to a negative local minimum of $J_\lambda$ and a solution which
corresponds to a positive mountain pass level of $J_\lambda$.

{\it Step 1.-}  {\it $J_\lambda$ has a local minimum $u_0$, such that $J_\lambda(u_0)<0$.}

We use the ideas in \cite{gp} to solve problems with concave-convex semilinear nonlinearities.

Consider  $\lambda_0>0 $ such that, if $0 < \lambda < \lambda_0$, $g$ attaints its positive maximum at $ r_{max}>0$. Take $r_0$ the lower
positive zero of $g$ and $r_0<r_1<r_{max}<r_2$ such that $g(r_1)>0$, $g(r_2)>0$, where $g$ is defined by \eqref{g}. Now consider a cutoff function
 $$ \tau : \mathbb{R}_+ \rightarrow [0,1],$$
such that $ \tau $ is nonincreasing,   $ \tau \in\mathcal{C}^{\infty} $ and it verifies
$$
\begin{cases}
&\tau (s)=1 \quad \text{if}\quad s \le r_0, \\ &\tau (s)=0 \quad \text{if}\quad s \ge r_1.
\end{cases}
$$
Let $ \Theta (u) = \tau ( \| \Delta u \|_2) $. We consider the truncated functional
\begin{equation}\label{truncado}
F_\lambda(u) = \displaystyle \frac12\int_\Omega |\Delta u|^2 \, dx -\int_\Omega u_{x_1} u_{x_2} u_{x_1 x_2}\Theta (u) \, dx - \lambda
\int_\Omega f u \, dx.
\end{equation}
As above, by H\"older and Sobolev inequalities we see $F(u) \ge h ( \| \Delta u \|_2 ) $, with
$$ h(s)=  \frac 12 \, s^2 - c_1 \, \tau(s) \, s^3 - \lambda \, ||f||_1 \, c_2 \, s.$$
The principal properties of $F$ defined by \eqref{truncado} are listed below.

\begin{lemma}\label{min}
\begin{enumerate}
\item $F_\lambda$ has the same regularity as $J_\lambda$. \item If $F_\lambda(u) < 0 $, then $ \| \Delta u \|_2 <
    r_0 $ , and $F_\lambda(u) = J_\lambda(u) $  if  $ \| \Delta u \|_2 <r_0$. \item Let $m$ be defined by
    $m=\inf\limits_{v\in W^{2,2}_0(\Omega)}F_\lambda(v)$.
\end{enumerate}
Then  $F_\lambda$ verifies a local Palais-Smale condition to the level $m$.
\end{lemma}

\begin{proof}
1) and 2) are immediate. To prove 3), observe that all Palais-Smale sequences of minimizers of  $F_\lambda$,  since $ m< 0 $, must be
bounded.
 Then by Lemma \ref{lemma:PS} we conclude.
\end{proof}
Observe that, by 2), if we find some negative critical value for $F_\lambda$, then we have that $m$ is a negative critical value  of
$J_\lambda$ and  there exist $u_0$ local minimum for $J_\lambda$.

\

{\it Step 2.-} {\it If $\lambda$ is small enough,  $J_\lambda$ has a {\it mountain pass} critical point, $u_*$, such that
$J_\lambda(u_*)>0$.}

By the estimates in subsection \ref{geom},  $J_\lambda$ verifies the geometrical requirements of the Mountain Pass Theorem (see \cite{AR} and
\cite{AE}). Consider $u_0$ the local minimum such that $J_\lambda(u_0)<0$ and consider $v\in W^{2,2}_0(\Omega)$ with $||\Delta v||_2>r_{max}$
and such that $J_\lambda(v)<J_\lambda(u_0)$. We define
$$\Gamma=\{\gamma\in \mathcal{C} \left([0,1], W^{2,2}_0(\Omega) \right)\,|\, \gamma(0)=u_0,\, \gamma(1)=v\},$$ and the minimax value
$$c=\inf_{\gamma\in\Gamma} \max_{t\in[0,1]} J_\lambda[\gamma(t)].$$
Applying the Ekeland variational principle (see \cite{E}), there exists a Palais-Smale sequence   to the level $c$, i.~e. there exists
$\{u_k\}_{k\in \mathbb{N}}\subset W^{2,2}_0(\Omega)$ such that
\begin{enumerate}
\item $J_\lambda(u_k)\to c$ as $k\to\infty$, \item $J_\lambda'(u_k)\to 0$ in $W^{-2,2}$.
\end{enumerate}

{\it Claim.-} {\it If $\{u_k\}_{k\in \mathbb{N}}\subset W^{2,2}_0(\Omega)$ is a Palais-Smale sequence for $J_\lambda $ at the level $c$, then
there exists $C>0$ such that $||\Delta u_k||_2< C$.}

Since the results in section~\ref{preliminar} hold then if $u\in W^{2,2}_0(\Omega)$, integrating by parts we find that
\begin{equation}\begin{array}{rcl}
\displaystyle \int_\Omega  u \, \text{det} \left( D^2 u \right) \, dx &=& \displaystyle \int_\Omega u      \left[ (u_{x_1} \, u_{x_2
x_2})_{x_1}-(u_{x_1} \, u_{x_2 x_1})_{x_2} \right] dx \\ \quad \\ &-& \displaystyle \int_\Omega (u_{x_1})^2 \, u_{x_2 x_2} \, dx +
\int_\Omega u_{x_1} \, u_{x_2 x_1} \, u_{x_2} \, dx =\\ \quad \\ && \displaystyle 2 \int_\Omega u_{x_1} \, u_{x_1 x_2} \, u_{x_2} \,
dx+\int_\Omega u_{x_1} \, u_{x_2 x_1} \, u_{x_2} \, dx  =\\ \quad
\\ && \displaystyle  3\int_\Omega u_{x_1} \, u_{x_2 x_1} \, u_{x_2} \, dx.
\end{array}
\end{equation}
Then if $\{u_k\}_{k\in \mathbb{N}}\subset W^{2,2}_0(\Omega)$ is a Palais-Smale sequence for $J_\lambda $ at the level $c$ and calling  $
\langle y_k, u_k\rangle= \langle J_\lambda '(u_k), u_k\rangle$
$$ \begin{array}{ll}
\displaystyle c+o(1)=J_\lambda (u_k)- \frac 13 \, \langle J_\lambda '(u_k), u_k\rangle+ \frac 13 \, \langle y_k, u_k\rangle \ge
\\ \ge \displaystyle \left(\frac 12-\frac 13 \right) \int_\Omega |
\Delta u_k|^2 \, dx- \frac 13 \, ||y_k||_{H^{-2}} \left(\int_\Omega | \Delta u_k|^2 \, dx \right)^{\frac 12} \\ \displaystyle - \frac 23 \,
\lambda \, C_S \, ||f||_{L^1} \left(\int_\Omega | \Delta u_k|^2 \, dx \right)^{\frac 12} ,
\end{array}
$$
where $C_S$ is a suitable Sobolev constant. This inequality implies that the sequence is bounded.

By using Lemma \ref{lemma:PS}, $J_\lambda$ satisfies the Palais-Smale condition to the level $c$. Therefore
\begin{enumerate}
\item $J_\lambda(u_*) = \lim\limits_{k\to\infty} J_\lambda(u_k)=c$ (and then $u_*$ is different from the local minimum, as in this case
    the
    value of the functional at this point is positive while in the other one was negative). \item $J'_\lambda(u_*)=0$, thus
$$\Delta^2 u_*=\det(D^2 u_*)+ \lambda f,\,\, u_*\in W^{2,2}_0(\Omega).$$
\end{enumerate}
In other words $u_*$ is a {\it mountain pass type} solution to the problem \eqref{ProD}.
\end{proof}

\begin{remark} Notice  that  we cannot directly conclude that a  bounded Palais-Smale sequence gives a solution  in the
distributional sense; indeed, we would need the convergence property $$\det(D^2 u_k)\rightharpoonup \det(D^2 u_*) \hbox{ at least  in }
L^1(\Omega).$$ To have this property up to passing to a subsequence we need almost everywhere convergence (see the result by Jones and
Journ\'e in \cite{jones-journe}). Notice that a. e. convergence for the second derivatives is only known after the proof of Lemma
\ref{lemma:PS}.
\end{remark}

\section{Some existence results including Navier boundary conditions}
\label{nbc}

We will find a solution to our problem with Navier boundary conditions
\begin{equation}\label{Pro2}
\left\{\begin{array}{rcl} \Delta^2 u&=&\text{det} \left( D^2 u \right) +\lambda f, \qquad x\in \Omega\subset\mathbb{R}^2, \\ u=0, &\,& \Delta
u =0\hbox{ on }\partial \Omega,
\end{array}
\right.
\end{equation}
and also with Dirichlet boundary conditions
\begin{equation}\label{Pro20}
\left\{\begin{array}{rcl} \Delta^2 u&=&\text{det} \left( D^2 u \right) +\lambda f, \qquad x\in \Omega\subset\mathbb{R}^2, \\ u=0, &\,&
\dfrac{\partial u}{\partial n}=0\hbox{ on }\partial \Omega.
\end{array}
\right.
\end{equation}

In this section we will prove the existence of at least one solution to problems~(\ref{Pro2}) and (\ref{Pro20}) by means of fixed point
methods.

First of all we need the following technical result

\begin{lemma}\label{lemmanaviert}
For any functions $v_1, v_2 \in W^{1,2}(\Omega)$ and $v_3 \in W_0^{1,2}(\Omega) \cap W^{2,2}(\Omega)$ the following equality is fulfilled
\begin{equation}\label{lemmanavier}
\int \text{det} \left(\nabla v_1, \nabla v_2 \right) v_3 \, dx = \int v_1 \, \nabla v_2 \cdot \nabla^\perp v_3 \, dx,
\end{equation}
where $\nabla^\perp v_3 = \left( \partial_{x_2} v_3, -\partial_{x_1} v_3 \right)$.
\end{lemma}

\begin{proof}
By Sobolev embedding we know $v_3$ is bounded in $L^\infty(\Omega)$ and consequently the left hand side of~(\ref{lemmanavier}) is well
defined. Now we take this expression and operate
\begin{eqnarray}\nonumber
\int \text{det} \left(\nabla v_1, \nabla v_2 \right) v_3 \, dx = \int \nabla \cdot \left( v_1 \partial_{x_2} v_2, -v_1
\partial_{x_1} v_2 \right) v_3 \, dx = \\
- \int \left( v_1 \partial_{x_2} v_2, -v_1
\partial_{x_1} v_2 \right) \cdot \nabla v_3 \, dx = \int v_1 \, \text{det}
\left( \nabla v_2, \nabla v_3 \right) \, dx \\ \nonumber = \int v_1 \, \nabla v_2 \cdot \nabla^\perp v_3 \, dx.
\end{eqnarray}
The first equality is obviously correct for smooth functions, and its validity can be extended by approximation to functions $v_1, v_2 \in
W^{1,2}(\Omega)$ and $v_3 \in C^1_0(\overline{\Omega})$ by considering the divergence on the right hand side in the distributional
sense~\cite{muller}. The second equality is a consequence of the definition of weak derivative and the fact that $v_3$ is a traceless
function. In this moment we can conclude both equalities are valid for $v_3 \in W^{1,2}_0(\Omega) \cap W^{2,2}(\Omega)$ because
$C^1_0(\overline{\Omega})$ is dense in $W^{1,2}_0(\Omega)$. The third and fourth equalities come from simple manipulations of the integrands.
Sobolev embedding guarantees that the last three terms in this chain of equalities are well defined.
\end{proof}
Now we move to prove the main result of this section

\begin{theorem}\label{theornavier}
If $\lambda >0$ is small enough then:
\begin{itemize}
\item[{a)}] There exists  $u \in W_0^{1,2}(\Omega) \cap W^{2,2}(\Omega)$ solution to  problem~(\ref{Pro2}). \item[{b)}] There exists
    $u\in W_0^{2,2}(\Omega)$ solution to problem~(\ref{Pro20}).
\end{itemize}
\end{theorem}
\begin{proof}
As the proof is similar in both cases  we skip the details of the case $b)$.

We start considering the linear problems
\begin{equation}\label{17}
\left\{\begin{array}{rcl} \Delta^2 u_1&=&\text{det} \left( D^2 \varphi_1 \right) +\lambda f, \qquad x\in \Omega\subset\mathbb{R}^2, \\ u_1=0,
&\,& \Delta u_1 =0\hbox{ on }\partial \Omega,
\end{array}
\right.
\end{equation}
and
\begin{equation}\label{17mas1}
\left\{\begin{array}{rcl} \Delta^2 u_2&=&\text{det} \left( D^2 \varphi_2 \right) +\lambda f, \qquad x\in \Omega\subset\mathbb{R}^2, \\ u_2=0,
&\,& \Delta u_2 =0\hbox{ on }\partial \Omega,
\end{array}
\right.
\end{equation}
where $\varphi_1, \varphi_2 \in W_0^{1,2}(\Omega) \cap W^{2,2}(\Omega)$. Classical results guarantee the existence of weak solution to both
problems in $W_0^{1,2}(\Omega) \cap W^{2,2}(\Omega)$ for given $\varphi_1$ and $\varphi_2$. Subtracting both equations we find
\begin{equation}\label{subtraction}
\left\{\begin{array}{rcl} \Delta^2 \left( u_1-u_2 \right) &=& \text{det} \left( D^2 \varphi_1 \right) -\text{det} \left( D^2 \varphi_2
\right), \qquad x\in \Omega\subset\mathbb{R}^2, \\ u_1-u_2=0, &\,& \Delta \left( u_1-u_2 \right) =0\hbox{ on }\partial \Omega.
\end{array}
\right.
\end{equation}
Now we note
\begin{eqnarray}
\nonumber \text{det} \left( D^2 \varphi_1 \right) -\text{det} \left( D^2 \varphi_2 \right) &=& \text{det} \{ \nabla (\varphi_1)_{x_1} ,
\nabla [(\varphi_1)_{x_2} - (\varphi_2)_{x_2}] \} \\ &+& \text{det} \{\nabla [(\varphi_1)_{x_1} - (\varphi_2)_{x_1}], \nabla
(\varphi_2)_{x_2} \}. \label{equalityb}
\end{eqnarray}
Next we see that for any $w \in W_0^{1,2}(\Omega) \cap W^{2,2}(\Omega)$ we have
\begin{eqnarray}
\left\langle \text{det} \left( D^2 \varphi_1 \right),w \right\rangle - \left\langle \text{det} \left( D^2 \varphi_2 \right),w \right\rangle
=
\\ \nonumber \int \left[
\partial_{x_1} \varphi_1 \nabla(\partial_{x_2} \varphi_1 - \partial_{x_2} \varphi_2) \cdot \nabla^\perp w \right] dx \\ \nonumber - \int
\left[
\partial_{x_2} \varphi_2 \nabla(\partial_{x_1} \varphi_1 -
\partial_{x_1} \varphi_2) \cdot
\nabla^\perp w \right] dx  \\ =\int \left[ \partial_{x_1} \varphi_1 \nabla(\partial_{x_2} \varphi_1 -
\partial_{x_2} \varphi_2) - \partial_{x_2} \varphi_2 \nabla(\partial_{x_1} \varphi_1 -
\partial_{x_1} \varphi_2) \right] \cdot \nabla^\perp w \,\, dx,
\nonumber
\end{eqnarray}
where we have used equality~(\ref{equalityb}) together with Lemma~\ref{lemmanaviert}. From these equalities we have the following chain of
inequalities
\begin{eqnarray}\label{estimateb}
\left| \left\langle \text{det} \left( D^2 \varphi_1 \right),w \right\rangle - \left\langle \text{det} \left( D^2 \varphi_2 \right),w
\right\rangle \right| \le \\ \nonumber \int (|\nabla \varphi_1|+|\nabla \varphi_2|) \, \left| D^2 (\varphi_1 - \varphi_2) \right| \, |\nabla
w| \, dx \le \\ \nonumber \left( ||\nabla \varphi_1||_4 + ||\nabla \varphi_2||_4 \right) \, \left|\left| D^2(\varphi_1 - \varphi_2)
\right|\right|_2 \, ||\nabla w||_4.
\end{eqnarray}
Now we take equation~(\ref{subtraction}), we multiply it by $(u_1-u_2)$ and integrate by parts the left hand side twice to find
\begin{eqnarray}
\int \left| \Delta (u_1 - u_2) \right|^2 \, dx = \\ \nonumber \left| \left\langle \text{det} \left( D^2 \varphi_1 \right) - \text{det} \left(
D^2 \varphi_2 \right), u_1 - u_2 \right\rangle \right| \le \\ \nonumber \left( ||\nabla \varphi_1||_4 + ||\nabla \varphi_2||_4 \right) \,
\left|\left| D^2(\varphi_1 - \varphi_2) \right|\right|_2 \, ||\nabla (u_1 - u_2)||_4 \le \\ \nonumber C \left( ||\nabla \varphi_1||_4 +
||\nabla \varphi_2||_4 \right) \, \left|\left| \Delta(\varphi_1 - \varphi_2) \right|\right|_2 \, ||\Delta (u_1-u_2)||_2,
\end{eqnarray}
where we have used Sobolev and Poincar\'e inequalities and the embedding of the $L^p$ spaces on bounded domains on the last step, and
result~(\ref{estimateb}) on the previous one. Simplifying the last chain of inequalities we arrive at
\begin{equation}
\left|\left|\Delta(u_1-u_2)\right|\right|_2 \le C \left( ||\nabla \varphi_1||_4 + ||\nabla \varphi_2||_4 \right) \, \left|\left|
\Delta(\varphi_1 - \varphi_2) \right|\right|_2
\end{equation}
for some suitable constant $C$.

Now, by using the Sobolev embedding $||\nabla \varphi_i||_4 \le C ||\Delta \varphi_i||_2$ for $i=1,2$, we find
\begin{equation}\label{24}
\left|\left|\Delta(u_1-u_2)\right|\right|_2 \le C \left( ||\Delta \varphi_1||_2 + ||\Delta \varphi_2||_2 \right) \, \left|\left|
\Delta(\varphi_1 - \varphi_2) \right|\right|_2.
\end{equation}

Consider $v$ the solution to the problem
\begin{equation}\label{lineal}
\left\{
\begin{array}{rcl}
\Delta^2 v&=&\lambda f, \qquad x\in \Omega\subset\mathbb{R}^2, \\ v=0, &\,& \Delta v =0\hbox{ on }\partial \Omega.
\end{array}
\right.
\end{equation}
Notice that
$$||\Delta v||_2\leq \lambda ||f||_1,$$
and then the first norm is small when $\lambda$ is small.

If $\varphi_i\in B_\rho(v)=\{ \varphi\in W^{2,2}(\Omega)\cap W_0^{1,2}(\Omega)\,|\, ||\Delta(v-\varphi)||_2\le \rho\,\}$, then \eqref{24}
becomes
\begin{equation}\label{25}
\begin{array}{rcl}
\left|\left|\Delta(u_1-u_2)\right|\right|_2 &\le& 2 C \left( \rho+||\Delta v||_2 \right)\left|\left| \Delta(\varphi_1 - \varphi_2)
\right|\right|_2\\ \le 2 C \big(\rho&+&\lambda ||f||_1 \big) \left|\left| \Delta(\varphi_1 - \varphi_2) \right|\right|_2\le \dfrac 12
\left|\left| \Delta(\varphi_1 - \varphi_2) \right|\right|_2
\end{array}
\end{equation}
for suitable $\lambda$ and $\rho$. Moreover, if $\varphi_i\in B_\rho(v)$  and $u_i$ is the corresponding solution to either problem
\eqref{17} or \eqref{17mas1}, we find that
$$||\Delta(u_i-v)||_2^2\le \int_\Omega \left| \text{det} \left( D^2 \varphi_i \right) \right|(u_i-v) dx\le
S\int_\Omega \left| \text{det} \left( D^2 \varphi_i \right) \right| dx\, ||\Delta(u_i-v)||_2,$$ that is
$$||\Delta(u_i-v)||_2\le S\int_\Omega \left| \text{det} \left( D^2 \varphi_i \right) \right| dx\le C||\Delta \varphi_i||_2^2.$$
But we can compute that
$$||\Delta \varphi_i||_2^2\le 2(||\Delta (\varphi_i-v)||_2^2+||\Delta v||_2^2)\le 2(\rho^2+\lambda^2||f||_1^2)\le \frac{\rho}{C}$$
for $\rho$ and $\lambda$ small enough. Thus
\begin{equation}\label{banach}
\left|\left|\Delta(u_1-u_2)\right|\right|_2 \le 2 C \left(\rho+\lambda ||f||_1 \right) \left|\left| \Delta(\varphi_1 - \varphi_2)
\right|\right|_2.
\end{equation}
For $\varphi\in B_\rho(v)$, define $u$ as the solution to
$$\left\{\begin{array}{rcl} \Delta^2 u&=&\text{det} \left( D^2 \varphi \right) +\lambda f, \qquad x\in \Omega\subset\mathbb{R}^2, \\
u=0, &\,& \Delta u =0\hbox{ on }\partial \Omega,
\end{array}
\right.$$ and then the nonlinear operator
$$\begin{array}{rcl}
T: B_\rho(v)&\rightarrow& B_\rho(v)\\ \varphi &\rightarrow&  T(\varphi)=u.
\end{array}
$$
By using the estimates above and the classical Banach fixed point theorem, we find a unique fixed point $u=T(u)$, which is a solution to
problem \eqref{Pro2}.
\end{proof}

\begin{corollary}
The solution to problem~(\ref{Pro2}) (respectively to problem (\ref{Pro20})) is unique in the ball $B_\rho(v)\subset W^{2,2}(\Omega)\cap
W^{1,2}_0(\Omega)$.
\end{corollary}

\begin{remark}
Notice that also the inhomogeneous Dirichlet problem can be solved by fixed point arguments. However in the homogenous Dirichlet problem with
this kind  of analysis we only find the trivial solution and not the mountain-pass type solution.
\end{remark}
\section{Further remarks and open problems}\label{further}
In this section we collect some further results that appear in the modelization by considering a more general functional than
\eqref{funcional de partida}, and performing different truncation arguments. We also quote some results in \cite{Los4} about the radial
setting and propose some open problems.

\subsection{Some insights from the radial case}

In \cite{Los4}, among other results, there  appears a numerical estimate  of the value of $\lambda$, or equivalently the size of the datum,
for which we have solvability of the  radial Dirichlet and Navier problems. The methods used are  those of the dynamical systems theory and a
shooting method of Runge-Kutta type.

It seems to be an open problem proving the existence of the corresponding threshold for the existence in general smooth domains.

\subsection{A subcritical quasilinear problem}
\label{subcriticalquasilinear} Let
\begin{equation}\label{funcional de partida2}
\mathcal{V}(u)= \int_\Omega \left(K_0+ K_1 H + \frac{K_2}{2} H^2 +\frac {K_3}{6} H^3\right)\sqrt{1+|\nabla u|^2}\, dx,
\end{equation}
we can yet select another problem that appears  by considering in the modelization the  functional (\ref{funcional de partida2}) with
$K_0=K_2=0$. The resulting equation is
\begin{equation}\label{qed}
\left\{\aligned \Delta (|\Delta u|^2) &= \quad\text{det} \, \left( D^2 \, u \right) + \lambda \, f \quad \text{in} \quad \Omega\subset
\mathbb{R}^2,\\ \left.u\right|_{\partial \Omega} &=0,\,\,\left.\frac{\partial u}{\partial n}\right|_{\partial \Omega}=0.
\endaligned
\right.
\end{equation}
The main difficulty associated with this problem is that it is not elliptic in general. This is the same problem that appears for functional
\eqref{funcional de partida2} when we assume that $K_2=K_3=0$. Then the Euler-Lagrange equation is
$$\Delta u+\det( D^2 u)=f$$
and performing the change of variable $v=u+\dfrac{|x|^2}{2}$ (personal advise of N. Trudinger) we obtain the Monge-Amp\`{e}re equation
$$\det(D^2 v)=f+1,$$
which is elliptic only if $f+1\ge 0$.

It seems to be an open problem finding the condition on the data $f$ for the solvability of problem \eqref{qed}.

A  problem formally close to \eqref{qed} is the following,
\begin{equation}\label{qed1}
\left\{\aligned \Delta \left(|\Delta u| \, \Delta u \right) &= \quad\text{det} \, \left( D^2 \, u \right) + \lambda \, f \quad \text{in}
\quad \Omega\subset \mathbb{R}^2,\\ \left.u\right|_{\partial \Omega} &=0,\,\,\left.\frac{\partial u}{\partial n}\right|_{\partial \Omega}=0,
\endaligned
\right.
\end{equation}
which is  elliptic. Despite it is a quasilinear problem, it is subcritical and then the variational formulation is easier. Indeed the energy
functional is given by
\begin{equation}\label{qenergy}
J(u)=\frac{1}{3} \int_\Omega |\Delta u|^3 \, dx-\int_\Omega u_{x_1} \, u_{x_2} \, u_{x_1 x_2} \, dx - \lambda \int f \, u \, dx,
\end{equation}
and it is defined in $W^{2,3}_0(\Omega)$. In this case, finding critical points could be done by means of the same arguments as before and
$f\in L^1(\Omega)$. For zeroth order nonlinearities this program was carried out in~\cite{BGP}.

As in former cases, the problem
\begin{equation}\label{qen}
\left\{\aligned \Delta \left( |\Delta u| \, \Delta u \right) &= \text{det} \, \left( D^2 \, u \right) + \lambda \, f \quad \text{in} \quad
\Omega\subset \mathbb{R}^2, \\ \left.u\right|_{\partial \Omega} &=0,\,\,\left.\Delta u\right|_{\partial \Omega} =0.
\endaligned
\right.
\end{equation}
can be   solved  using a convenient fixed point argument.

The variational setting does not work for the same reasons as before, namely that the critical points of the functional do not correspond to
solutions of our problem.

\subsection{An extension of the Kardar-Parisi-Zhang equation with Navier boundary conditions}
\label{kpznavier}
 A nonvariational high order problem which is the  counterpart of the Kardar-Parisi-Zhang equation in order 2,
is the following problem
\begin{equation}\label{kpz4}
\left\{\aligned \Delta^2 u&=|\nabla(\Delta u)|^2   + \lambda \, f \quad \text{in} \quad \Omega\subset \mathbb{R}^N, \\
\left.u\right|_{\partial \Omega} &=0,\,\,\left.\Delta u\right|_{\partial \Omega} =0.
\endaligned
\right.
\end{equation}
Notice that herein we will quote results that are valid for any arbitrary spatial dimension $N$. We assume that $f\in L^m(\Omega)$ with $m\ge
\frac N2$, $f(x)\ge 0$.

We call $-\Delta u=v$. Then we find the equivalent system given by
\begin{equation}\label{system}
\left\{\aligned -\Delta u &= v    \quad \text{in} \quad \Omega, \quad  \left.u\right|_{\partial \Omega} =0, \\ -\Delta v&=|\nabla
v|^2+\lambda f \quad \text{in} \quad \Omega,\quad \left.v\right|_{\partial \Omega} =0.
\endaligned
\right.
\end{equation}
The second equation has been studied extensively  (see for instance \cite{ireneo1} and \cite{ireneo2} for the parabolic case) where necessary
and sufficient conditions for the existence of solutions are established; moreover, a complete characterization of the solutions is
presented. In particular if $v\in H^1_0(\Omega)$ is a solution to the second equation of \eqref{system}, then it verifies that
$$e^{\delta v}-1\in H^1_0(\Omega),\hbox{   for all    } \delta\in \left( 0,\frac 12 \right),$$
and for $\delta=\frac 12$ we reach a bounded solution. Then we obtain the following result for free.
\begin{theorem}\label{th:kpz4}
Problem \eqref{kpz4} has infinitely many solutions in $W^{3,2}(\Omega)$ such that if $-\Delta u=v$ then
$$e^{\delta v}-1\in H^1_0(\Omega),\hbox{   for all    } \delta\in \left( 0,\frac 12 \right).$$
If $\delta= \frac 12$ then $u\in W^{3, p}(\Omega)$ for all $p<\infty$.
\end{theorem}
Notice that just the regularity of the solution gives sense to the second member of the equation defining boundary value
problem~(\ref{kpz4}).

On the contrary the equation in \eqref{kpz4} with Dirichlet data,  provides an interesting set of  open problems.

\subsection{Other  boundary conditions}\label{Neumann}
It would be interesting to analyze the Neumann problem, that is,
\begin{equation}\label{neumann}
\left\{\aligned \Delta^2 u  &= \quad\text{det} \, \left( D^2 \, u \right) + \lambda \, f \quad \text{in} \quad \Omega\subset \mathbb{R}^2,\\
\left.\dfrac{\partial u}{\partial n}\right|_{\partial \Omega} &=0,\,\,\left.\frac{\partial \Delta u}{\partial n}\right|_{\partial \Omega}=0.
\endaligned
\right.
\end{equation}
Also it seems to be interesting to analyze inhomogeneous boundary conditions and, perhaps the more interesting  from the physical view point,
 periodic boundary conditions.

We leave  this analysis for the future.

\end{document}